\documentclass{amsart}
\usepackage{amsmath}
\usepackage{amssymb}
\usepackage{amsfonts}
\usepackage{graphicx}

\makeatletter \@namedef{subjclassname@2010}{  \textup{2010}
Mathematics Subject Classification}
\theoremstyle{plain}
\newtheorem{theorem}{Theorem}[section]
\newtheorem{corollary}[theorem]{Corollary}
\newtheorem{proposition}[theorem]{Proposition}
\newtheorem{lemma}[theorem]{Lemma}
\theoremstyle{remark}
\newtheorem{remark}[theorem]{Remark}
\theoremstyle{definition}
\theoremstyle{example}
\newtheorem{example}[theorem]{Example}

\numberwithin{equation}{section}
\input{tcilatex}

\begin{document}
\title[OPERATORS WITH\ A SINGLETON\ LOCAL SPECTRUM]{LOCAL\ SPECTRUM, LOCAL\
SPECTRAL RADIUS, AND GROWTH CONDITIONS }
\author{HEYBETKULU\ MUSTAFAYEV}
\address{Van Yuzuncu Yil University, Faculty of Science, Department of
Mathematics, VAN-TURKEY}
\email{hsmustafayev@yahoo.com}

\begin{abstract}
Let $X$ be a complex Banach space and $x\in X.$ Assume that a bounded linear
operator $T$ on $X$ satisfies the condition 
\begin{equation*}
\left\Vert e^{tT}x\right\Vert \leq C_{x}\left( 1+\left\vert t\right\vert
\right) ^{\alpha }\text{\ }\left( \alpha \geq 0\right) ,
\end{equation*}%
for all $t\in 
\mathbb{R}
$ and for some constant $C_{x}>0.$ For the function $f$ from the Beurling
algebra $L_{\omega }^{1}\left( 
\mathbb{R}
\right) $ with the weight $\omega \left( t\right) =\left( 1+\left\vert
t\right\vert \right) ^{\alpha },$ we can define an element in $X$, denoted
by $x_{f}$, which integrates $e^{tT}x$ with respect to $f.$ We present
complete description of the elements $x_{f}$ in the case when the local
spectrum of $T$ at $x$ consists of one-point. In the case $0\leq \alpha <1,$
some estimates for the norm of $Tx$ via local spectral radius of $T$ at $x$
are obtained. Some applications of these results are also given.
\end{abstract}

\maketitle

\section{Introduction}

Throughout the paper, $X$ will denote a complex Banach space and $B\left(
X\right) ,$ the algebra of all bounded linear operators on $X.$ As usual,
the spectrum and the spectral radius of $T\in B\left( X\right) $ will be
denoted by $\sigma \left( T\right) $ and $r\left( T\right) ,$ respectively.
A classical theorem of Gelfand \cite{7} states that if $\sigma \left(
T\right) =\left\{ 1\right\} $ and $\sup_{n\in 
\mathbb{Z}
}\left\Vert T^{n}\right\Vert <\infty ,$ then $T=I.$ A result more general
than Gelfand's theorem has been proved by Hille \cite{9}: If $\sigma \left(
T\right) =\left\{ 1\right\} $ and $\left\Vert T^{n}\right\Vert =O\left(
\left\vert n\right\vert ^{k}\right) $ $\left( \left\vert n\right\vert
\rightarrow \infty \right) ,$ then $\left( T-I\right) ^{k+1}=0.$ Over the
years, these results led to a number of extensions and variations, as shown
in the survey article by Zemanek \cite{21}.

Continuous version of the Gelfand theorem states that if a quasinilpotent
operator $T\in B\left( X\right) $ satisfies the condition $\sup_{t\in 
\mathbb{R}
}\left\Vert e^{tT}\right\Vert <\infty $, then $T=0.$ We give a sketch of the
proof:

For an arbitrary $f\in L^{1}\left( 
\mathbb{R}
\right) ,$ we can define a bounded linear operator $T_{f}$ on $X$ by 
\begin{equation*}
T_{f}x=\int_{%
\mathbb{R}
}f\left( t\right) e^{tT}xdt\text{, \ }x\in X.
\end{equation*}%
Then the map $f\mapsto T_{f}$ is a continuous homomorphism. It is easy to
check that $I:=\left\{ f\in L^{1}\left( 
\mathbb{R}
\right) :T_{f}=0\right\} $ is a closed ideal of $L^{1}\left( 
\mathbb{R}
\right) $ whose hull is $\left\{ 0\right\} .$ Since $\left\{ 0\right\} $ is
a set of synthesis for $L^{1}\left( 
\mathbb{R}
\right) ,$ we have $I=\left\{ f\in L^{1}\left( 
\mathbb{R}
\right) :\widehat{f}\left( 0\right) =0\right\} ,$ where $\widehat{f}$ is the
Fourier transform of $f.$ It follows that for a fixed $x\in X,$ the
functions $t\rightarrow \varphi \left( e^{tT}x\right) $ are constants for
every $\varphi \in X^{\ast }.$ Consequently, $\varphi \left( e^{tT}x\right)
=\varphi \left( x\right) $ for every $\varphi \in X^{\ast }$ and for all $%
t\in 
\mathbb{R}
.$ So we have $e^{tT}x=x$ for all $x\in X$ and $t\in 
\mathbb{R}
.$ This clearly implies that $T=0.$

In this paper, we develop this idea further and prove some results for the
local spectrum and local spectral radius for a wider class of operators.

Let $L_{\omega }^{1}\left( 
\mathbb{R}
\right) $ be the Beurling algebra with the weight $\omega \left( t\right)
=\left( 1+\left\vert t\right\vert \right) ^{\alpha }$ $\left( \alpha \geq
0\right) $ and let $x\in X.$ Assume that $T\in B\left( X\right) $ satisfies
the condition 
\begin{equation}
\left\Vert e^{tT}x\right\Vert \leq C_{x}\left( 1+\left\vert t\right\vert
\right) ^{\alpha }\text{, }  \label{1.1}
\end{equation}%
for all $t\in 
\mathbb{R}
$ and for some constant $C_{x}>0.$ For an arbitrary $f\in L_{\omega
}^{1}\left( 
\mathbb{R}
\right) ,$ we can define an element $x_{f}$ in $X$ by 
\begin{equation}
x_{f}=\int_{%
\mathbb{R}
}f\left( t\right) e^{tT}xdt.  \label{1.2}
\end{equation}%
\ 

We present complete description of the elements $x_{f}$ in the case when the
local spectrum of $T$ at $x$ consists of one-point. In the case $0\leq
\alpha <1,$ some estimates for the norm of $Tx$ via local spectral radius of 
$T$ at $x$ are given. Some applications of these results are also given.

\section{The results}

Let us first introduce some basic definitions and results.

For an arbitrary $T\in B\left( X\right) $ and $x\in X$, we define $\rho
_{T}\left( x\right) $ to be the set of all $\lambda \in 
\mathbb{C}
$ for which there exists a neighborhood $U_{\lambda }$ of $\lambda $ with $%
u\left( z\right) $ analytic on $U_{\lambda }$ having values in $X$ such that 
$\left( zI-T\right) u\left( z\right) =x$ for all $z\in U_{\lambda }$. This
set is open and contains the resolvent set of $T$. By definition, the 
\textit{local spectrum} of $T$ at $x\in X$, denoted by $\sigma _{T}\left(
x\right) $ is the complement of $\rho _{T}\left( x\right) $, so it is a
compact subset of $\sigma \left( T\right) $. This object is most tractable
if the operator $T$ has the \textit{single-valued extension property }%
(SVEP), i.e. for every open set $U$ in $%
\mathbb{C}
,$ the only analytic function $u:U\rightarrow X$ for which the equation $%
\left( zI-T\right) u\left( z\right) =0$ holds is the constant function $%
u\equiv 0$. If $T$ has SVEP, then $\sigma _{T}\left( x\right) \neq \emptyset
,$ whenever $x\in X\diagdown \left\{ 0\right\} $ \cite[Proposition 1.2.16]%
{12}. It is easy to see that an operator $T\in B\left( X\right) $ having
spectrum without interior points has the SVEP.

The \textit{local spectral radius }of $T\in B\left( X\right) $ at $x\in X$
is defined by 
\begin{equation*}
r_{T}\left( x\right) =\sup \left\{ \left\vert \lambda \right\vert :\lambda
\in \sigma _{T}\left( x\right) \right\} .
\end{equation*}%
If $T$ has SVEP then%
\begin{equation*}
r_{T}\left( x\right) =\underset{n\rightarrow \infty }{\overline{\lim }}%
\left\Vert T^{n}x\right\Vert ^{\frac{1}{n}}
\end{equation*}%
\cite[Proposition 3.3.13]{12}.

Recall that a \textit{weight} \textit{function} (shortly a \textit{weight}) $%
\omega $ is a continuous function on $%
\mathbb{R}
$ such that 
\begin{equation*}
\omega \left( t\right) \geq 1\text{ \ and \ }\left( t+s\right) \leq \omega
\left( t\right) \omega \left( s\right) \text{, \ }\forall t,s\in 
\mathbb{R}
.
\end{equation*}%
For a weight function $\omega $, by $L_{\omega }^{1}\left( 
\mathbb{R}
\right) $ we will denote the Banach space of the functions $f\in L^{1}\left( 
\mathbb{R}
\right) $ with the norm%
\begin{equation*}
\left\Vert f\right\Vert _{1,\omega }=\int_{%
\mathbb{R}
}\left\vert f\left( t\right) \right\vert \omega \left( t\right) dt<\infty .
\end{equation*}%
The space $L_{\omega }^{1}\left( 
\mathbb{R}
\right) $ with convolution product and the norm $\left\Vert \cdot
\right\Vert _{1,\omega }$ is a commutative Banach algebra and is called 
\textit{Beurling algebra}. The dual space of $L_{\omega }^{1}\left( 
\mathbb{R}
\right) ,$ denoted by $L_{\omega }^{\infty }\left( 
\mathbb{R}
\right) $, is the space of all measurable functions $g$ on $%
\mathbb{R}
$ such that 
\begin{equation*}
\left\Vert g\right\Vert _{\omega ,\infty }:=\text{ess}\sup_{t\in 
\mathbb{R}
}\frac{\left\vert g\left( t\right) \right\vert }{\omega \left( t\right) }%
<\infty .
\end{equation*}%
The duality being implemented by the formula%
\begin{equation*}
\langle g,f\rangle =\int_{%
\mathbb{R}
}g\left( -t\right) f\left( t\right) dt\text{, \ }\forall f\in L_{\omega
}^{1}\left( 
\mathbb{R}
\right) ,\text{ }\forall g\in L_{\omega }^{\infty }\left( 
\mathbb{R}
\right) .
\end{equation*}%
Throughout, $\left[ \alpha \right] $ will denote the integer part of $\alpha
\in 
\mathbb{R}
.$

Let $x\in X$ and assume that $T\in B\left( X\right) $ satisfies the
condition (1.1). For an arbitrary $f\in L_{\omega }^{1}\left( 
\mathbb{R}
\right) ,$ we can define an element $x_{f}$ of $X$ by (1.2).

\begin{theorem}
Let $\omega \left( t\right) =\left( 1+\left\vert t\right\vert \right)
^{\alpha }$ $\left( \alpha \geq 0\right) .$ Assume that $T\in B\left(
X\right) $ has SVEP\ and $x\in X$ satisfies the condition $\left\Vert
e^{tT}x\right\Vert \leq C_{x}\omega \left( t\right) $ for all $t\in 
\mathbb{R}
$ and for some constant $C_{x}>0.$ Then the following assertions hold:

$\left( a\right) $ $\sigma _{T}\left( x\right) \subset i%
\mathbb{R}
.$

$\left( b\right) $ If $\sigma _{T}\left( x\right) =\left\{ i\lambda \right\} 
$ $\left( \lambda \in 
\mathbb{R}
\right) ,$ then for an arbitrary $f\in L_{\omega }^{1}\left( 
\mathbb{R}
\right) $, 
\begin{equation*}
x_{f}=\widehat{f}\left( \lambda \right) x+\frac{\widehat{f}^{\prime }\left(
\lambda \right) }{1!}\left( iT+\lambda I\right) x+...+\frac{\widehat{f}%
^{\left( k\right) }\left( \lambda \right) }{k!}\left( iT+\lambda I\right)
^{k}x,
\end{equation*}%
where $k=[\alpha ].$
\end{theorem}

For the proof, we need some preliminary results.

Recall that a weight function $\omega $ is is said to be \textit{regular} if 
\begin{equation*}
\int_{%
\mathbb{R}
}\frac{\log \omega \left( t\right) }{1+t^{2}}dt<\infty .
\end{equation*}%
For example, $\omega \left( t\right) =\left( 1+\left\vert t\right\vert
\right) ^{\alpha }$ $\left( \alpha \geq 0\right) $\ is a regular weight and
it is called \textit{polynomial weight}. If $\omega $ is a regular weight
(in this paper, we will consider regular weights only), then 
\begin{equation}
\lim_{t\rightarrow +\infty }\frac{\log \omega \left( t\right) }{t}%
=\lim_{t\rightarrow +\infty }\frac{\log \omega \left( -t\right) }{t}=0.
\label{2.1}
\end{equation}%
Consequently, the maximal ideal space of the algebra $L_{\omega }^{1}\left( 
\mathbb{R}
\right) $ can be identified with $%
\mathbb{R}
$. The Gelfand transform of $f\in L_{\omega }^{1}\left( 
\mathbb{R}
\right) $ is just the Fourier transform of $f.$ It follows that $L_{\omega
}^{1}\left( 
\mathbb{R}
\right) $ is semisimple. Moreover, the algebra $L_{\omega }^{1}\left( 
\mathbb{R}
\right) $ is regular (in the Shilov sense). Recall also that $L_{\omega
}^{1}\left( 
\mathbb{R}
\right) $ is Tauberian, that is, the set $\left\{ f\in L_{\omega }^{1}\left( 
\mathbb{R}
\right) :\text{supp}\widehat{f}\text{ is compact}\right\} $ is dense in $%
L_{\omega }^{1}\left( 
\mathbb{R}
\right) $ (see, \cite[Ch.III]{8}, \cite[Ch.5]{13}, and \cite[Ch.2]{17}).

We will need also the following well-known fact which is true for every
commutative semisimple regular Banach algebra:

If $\left\{ I_{\lambda }\right\} _{\lambda \in \Lambda }$ is a collection of
the closed ideals of $L_{\omega }^{1}\left( 
\mathbb{R}
\right) $, then 
\begin{equation}
\text{\textnormal{hull}}\left( \bigcap_{\lambda \in \Lambda }I_{\lambda
}\right) =\overline{\bigcup_{\lambda \in \Lambda }\text{\textnormal{hull}}%
\left( I_{\lambda }\right) }.  \label{2.2}
\end{equation}

Denote by $M_{\omega }\left( 
\mathbb{R}
\right) $ the Banach algebra (with respect to convolution product) of all
complex regular Borel measures on $%
\mathbb{R}
$ such that 
\begin{equation*}
\left\Vert \mu \right\Vert _{1,\omega }:=\int_{%
\mathbb{R}
}\omega \left( t\right) d\left\vert \mu \right\vert \left( t\right) <\infty .
\end{equation*}%
The algebra $L_{\omega }^{1}\left( 
\mathbb{R}
\right) $ is naturally identifiable with a closed ideal of $M_{\omega
}\left( 
\mathbb{R}
\right) .$ By $\widehat{f}$ and $\widehat{\mu }$, we will denote the Fourier
and the Fourier-Stieltjes transform of $f\in L_{\omega }^{1}\left( 
\mathbb{R}
\right) $ and $\mu \in M_{\omega }\left( 
\mathbb{R}
\right) $, respectively.

To any closed subset $S$ of $%
\mathbb{R}
,$ the following two closed ideals of $L_{\omega }^{1}\left( 
\mathbb{R}
\right) $ associated:%
\begin{equation*}
I_{\omega }\left( S\right) :=\left\{ f\in L_{\omega }^{1}\left( 
\mathbb{R}
\right) :\widehat{f}\left( S\right) =\left\{ 0\right\} \right\}
\end{equation*}%
and 
\begin{equation*}
J_{\omega }\left( S\right) :=\overline{\left\{ f\in L_{\omega }^{1}\left( 
\mathbb{R}
\right) :\text{supp}\widehat{f}\text{ is compact and supp}\widehat{f}\cap
S=\emptyset \right\} }.
\end{equation*}%
The ideals $J_{\omega }\left( S\right) $ and $I_{\omega }\left( S\right) $
are respectively, the smallest and the largest closed ideals in $L_{\omega
}^{1}\left( 
\mathbb{R}
\right) $ with hull $S$. When these two ideals coincide, the set $S$ is said
to be a \textit{set of} \textit{synthesis} for $L_{\omega }^{1}\left( 
\mathbb{R}
\right) $ (for instance, see \cite[Sect. 8.3]{11}). Notice that%
\begin{equation*}
I_{\omega }\left( \left\{ \infty \right\} \right) =L_{\omega }^{1}\left( 
\mathbb{R}
\right) \text{ \ and \ }J_{\omega }\left( \left\{ \infty \right\} \right) =%
\overline{\left\{ f\in L_{\omega }^{1}\left( 
\mathbb{R}
\right) :\text{supp}\widehat{f}\text{ is compact}\right\} }.
\end{equation*}%
Since the algebra $L_{\omega }^{1}\left( 
\mathbb{R}
\right) $ is Tauberian, $I_{\omega }\left( \left\{ \infty \right\} \right)
=J_{\omega }\left( \left\{ \infty \right\} \right) $. Hence, $\left\{ \infty
\right\} $ is a set of synthesis for $L_{\omega }^{1}\left( 
\mathbb{R}
\right) .$

If $\omega \left( t\right) =\left( 1+\left\vert t\right\vert \right)
^{\alpha }$, then the first $k$ $\left( =\left[ \alpha \right] \right) $
derivatives of the Fourier transform of $f\in L_{\omega }^{1}\left( 
\mathbb{R}
\right) $ exist at each point $\lambda \in 
\mathbb{R}
$ and 
\begin{equation*}
J_{\omega }\left( \left\{ \lambda \right\} \right) =\left\{ f\in L_{\omega
}^{1}\left( 
\mathbb{R}
\right) :\widehat{f}\left( \lambda \right) =\widehat{f}^{\prime }\left(
\lambda \right) =...=\widehat{f}^{\left( k\right) }\left( \lambda \right)
=0\right\}
\end{equation*}%
(for instance, see \cite[Ch. IV]{8} and \cite{20}). It follows that if $%
0\leq \alpha <1,$ then each point (consequently, each finite subset) of $%
\mathbb{R}
$ is a set of synthesis for $L_{\omega }^{1}\left( 
\mathbb{R}
\right) $.

Let $B_{g}$ be the weak$^{\ast }$-closed translation invariant subspace of $%
L_{\omega }^{\infty }\left( 
\mathbb{R}
\right) $ generated by $g\in L_{\omega }^{\infty }\left( 
\mathbb{R}
\right) .$ The \textit{Beurling spectrum }sp$_{B}\left( g\right) $ of $g\in
L_{\omega }^{\infty }\left( 
\mathbb{R}
\right) $ is defined by 
\begin{equation*}
\text{sp}_{B}\left( g\right) =\left\{ \lambda \in 
\mathbb{R}
:e^{-i\lambda t}\in B_{g}\right\} .
\end{equation*}%
It is easy to verify that 
\begin{equation*}
\text{sp}_{B}\left( g\right) =\text{hull}\left( \mathcal{I}_{g}\right) \text{%
,}
\end{equation*}%
where%
\begin{equation*}
\mathcal{I}_{g}=\left\{ f\in L_{\omega }^{1}\left( 
\mathbb{R}
\right) :f\ast g=0\right\}
\end{equation*}%
is a closed ideal of $L_{\omega }^{1}\left( 
\mathbb{R}
\right) $.

The \textit{Carleman transform} of $g\in L_{\omega }^{\infty }\left( 
\mathbb{R}
\right) $ is defined as the analytic function $G\left( z\right) $ on $%
\mathbb{C}
\diagdown i%
\mathbb{R}
$ given by 
\begin{equation*}
G\left( z\right) =\left\{ 
\begin{tabular}{l}
$\int\limits_{0}^{\infty }e^{-zt}{g\left( t\right) dt\text{, }\ }$%
\textnormal{Re}${z>0\text{;}}$ \\ 
${-}\int\limits_{-\infty }^{0}e^{-zt}{g\left( t\right) dt\text{, \ Re}z<0%
\text{.}}$%
\end{tabular}%
\right.
\end{equation*}%
We know \cite{10} that $\lambda \in $sp$_{B}\left\{ g\right\} $ if and only
if the Carleman transform $G\left( z\right) $ of $g$ has no analytic
extension to a neighborhood of $i\lambda .$

Let $\omega $ be a weight function, $T\in B\left( X\right) ,$ and let%
\begin{equation*}
E_{T}^{\omega }:=\left\{ x\in X:\exists C_{x}>0,\text{ }\left\Vert
e^{tT}x\right\Vert \leq C_{x}\omega \left( t\right) \text{, \ }\forall t\in 
\mathbb{R}
\right\} .
\end{equation*}%
Then, $E_{T}^{\omega }$ is a linear (non-closed, in general) subspace of $X$%
. If $x\in E_{T}^{\omega }$, then for an arbitrary $\mu \in M_{\omega
}\left( 
\mathbb{R}
\right) ,$ we can define an element $x_{\mu }$ in $X$ by 
\begin{equation*}
x_{\mu }=\int_{%
\mathbb{R}
}e^{tT}xd\mu \left( t\right) \text{.}
\end{equation*}%
Notice that $\mu \mapsto x_{\mu }$ is a bounded linear map from $M_{\omega
}\left( 
\mathbb{R}
\right) $ into $X$ and%
\begin{equation*}
\left\Vert x_{\mu }\right\Vert \leq C_{x}\left\Vert \mu \right\Vert
_{1,\omega }\text{, \ }\forall \mu \in M_{\omega }\left( 
\mathbb{R}
\right) .
\end{equation*}%
Further, from the identity%
\begin{equation*}
e^{tT}x_{\mu }=\int_{%
\mathbb{R}
}e^{\left( t+s\right) T}xd\mu \left( s\right) ,
\end{equation*}%
we can write 
\begin{eqnarray*}
\left\Vert e^{tT}x_{\mu }\right\Vert &\leq &\int_{%
\mathbb{R}
}\left\Vert e^{\left( t+s\right) T}x\right\Vert d\left\vert \mu \right\vert
\left( s\right) \\
&\leq &C_{x}\int_{%
\mathbb{R}
}\omega \left( t+s\right) d\left\vert \mu \right\vert \left( s\right) \\
&\leq &C_{x}\int_{%
\mathbb{R}
}\omega \left( t\right) \omega \left( s\right) d\left\vert \mu \right\vert
\left( s\right) \\
&=&C_{x}\left\Vert \mu \right\Vert _{1,\omega }\omega \left( t\right) \text{%
, \ }\forall t\in 
\mathbb{R}
.
\end{eqnarray*}%
This shows that if $x\in E_{T}^{\omega },$ then $x_{\mu }\in E_{T}^{\omega }$
for every $\mu \in M_{\omega }\left( 
\mathbb{R}
\right) .$ In other words, $E_{T}^{\omega }$ is a $M_{\omega }\left( 
\mathbb{R}
\right) -$module for the map $\left( \mu ,x\right) \mapsto x_{\mu }.$

It is easy to check that 
\begin{equation*}
\left( x_{\mu }\right) _{\nu }=x_{\mu \ast \nu }\text{, \ }\forall \mu ,\nu
\in M_{\omega }\left( 
\mathbb{R}
\right) .
\end{equation*}%
It follows that if $x\in E_{T}^{\omega },$ then 
\begin{equation*}
I_{x}:=\left\{ f\in L_{\omega }^{1}\left( 
\mathbb{R}
\right) :x_{f}=0\right\}
\end{equation*}%
is a closed ideal of $L_{\omega }^{1}\left( 
\mathbb{R}
\right) ,$ where%
\begin{equation*}
x_{f}=\int_{%
\mathbb{R}
}f\left( t\right) e^{tT}xdt\text{.}
\end{equation*}

For a given $x\in E_{T}^{\omega }$, consider the function 
\begin{equation}
u\left( z\right) :=\left\{ 
\begin{tabular}{l}
$\int\limits_{0}^{\infty }e^{-zt}e^{tT}x{dt\text{, \ }\func{Re}z>0\text{;}}$
\\ 
${-}\int\limits_{-\infty }^{0}e^{-zt}e^{tT}x{dt\text{, \ }\func{Re}z<0\text{.%
}}$%
\end{tabular}%
\right.  \label{2.3}
\end{equation}%
It follows from (2.1) that $u\left( z\right) $ is a function analytic on $%
\mathbb{C}
\diagdown i%
\mathbb{R}
$. Let $a:=\func{Re}z>0.$ Then, for an arbitrary $s>0$ we can write%
\begin{equation*}
\left( zI-T\right) \int\limits_{0}^{s}e^{-zt}e^{tT}x{dt}=-\int\limits_{0}^{s}%
\frac{d}{dt}e^{t\left( T-zI\right) }x{dt=x-}e^{s\left( T-zI\right) }x.
\end{equation*}%
Since 
\begin{equation*}
\left\Vert e^{s\left( T-zI\right) }x\right\Vert =e^{-as}\left\Vert
e^{sT}x\right\Vert \leq C_{x}e^{-as}\omega \left( s\right)
\end{equation*}%
and 
\begin{equation*}
\lim_{s\rightarrow +\infty }e^{-as}\omega \left( s\right) =0,
\end{equation*}%
we have 
\begin{equation*}
\left( zI-T\right) u\left( z\right) =x,\text{ \ for all }z\in 
\mathbb{C}
\text{ with }\func{Re}z>0.
\end{equation*}%
Similarly,%
\begin{equation*}
\left( zI-T\right) u\left( z\right) =x,\text{ \ for all }z\in 
\mathbb{C}
\text{ with }\func{Re}z<0.
\end{equation*}%
Hence, 
\begin{equation}
\left( zI-T\right) u\left( z\right) =x\text{, \ }\forall z\in 
\mathbb{C}
\diagdown i%
\mathbb{R}
.  \label{2.4}
\end{equation}%
This clearly implies that $\sigma _{T}\left( x\right) \subset i%
\mathbb{R}
.$

Thus we have the following:

\begin{proposition}
Let $\omega $ be a regular weight. Assume that $T\in B\left( X\right) $ and $%
x\in X$ satisfy the condition $\left\Vert e^{tT}x\right\Vert \leq
C_{x}\omega \left( t\right) $ for all $t\in 
\mathbb{R}
$ and for some constant $C_{x}>0.$ Then $\sigma _{T}\left( x\right) \subset i%
\mathbb{R}
.$
\end{proposition}

Now, assume that $T$ has SVEP. We claim that $\sigma _{T}\left( x\right) $
consists of all $\lambda \in i%
\mathbb{R}
$ for which the function $u\left( z\right) $ (which is defined by (2.3)) has
no analytic extension to a neighborhood of $\lambda $. To see this, let $%
v\left( z\right) $ be the analytic extension of $u\left( z\right) $ to a
neighborhood $U_{\lambda }$ of $\lambda \in i%
\mathbb{R}
.$ It follows from the identity (2.4) that the function%
\begin{equation*}
w\left( z\right) :=\left( zI-T\right) v\left( z\right) -x
\end{equation*}%
vanishes on $U_{\lambda }^{+}:=\left\{ z\in U_{\lambda }:\func{Re}%
z>0\right\} $ and on $U_{\lambda }^{-}:=\left\{ z\in U_{\lambda }:\func{Re}%
z<0\right\} .$ By uniqueness theorem, $w\left( z\right) =0$ for all $z\in
U_{\lambda }.$ So we have 
\begin{equation*}
\left( zI-T\right) v\left( z\right) =x\text{, \ }\forall z\in U_{\lambda }.
\end{equation*}%
This shows that $\lambda \in \rho _{T}\left( x\right) .$ If $\lambda \in
\rho _{T}\left( x\right) \cap i%
\mathbb{R}
,$ then there exists a neighborhood $U_{\lambda }$ of $\lambda $ with $%
v\left( z\right) $ analytic on $U_{\lambda }$ having values in $X$ such that 
\begin{equation*}
\left( zI-T\right) v\left( z\right) =x\text{, \ }\forall z\in U_{\lambda }.
\end{equation*}%
By (2.4), 
\begin{equation*}
\left( zI-T\right) \left( u\left( z\right) -v\left( z\right) \right) =0,%
\text{ \ }\forall z\in U_{\lambda }^{+},\text{ }\forall z\in U_{\lambda
}^{-}.
\end{equation*}%
Since $T$ has SVEP, we have 
\begin{equation*}
u\left( z\right) =v\left( z\right) \text{, \ }\forall z\in U_{\lambda }^{+},%
\text{ }\forall z\in U_{\lambda }^{-}.
\end{equation*}%
This shows that $u\left( z\right) $ can be analytically extended to a
neighborhood of $\lambda $.

Let $x\in E_{T}^{\omega }$ be given. For an arbitrary $\varphi \in X^{\ast
}, $ define a function $\varphi _{x}$ on $%
\mathbb{R}
$ by 
\begin{equation*}
\varphi _{x}\left( t\right) =\langle \varphi ,e^{tT}x\rangle .
\end{equation*}%
Then $\varphi _{x}$ is continuous and 
\begin{equation*}
\left\vert \varphi _{x}\left( t\right) \right\vert \leq C_{x}\left\Vert
\varphi \right\Vert \omega \left( t\right) \text{, \ }\forall t\in 
\mathbb{R}
.
\end{equation*}%
Consequently, $\varphi _{x}\in L_{\omega }^{\infty }\left( 
\mathbb{R}
\right) $. Taking into account the identity (2.3), we have 
\begin{equation*}
\langle \varphi ,u\left( z\right) \rangle =\left\{ 
\begin{tabular}{l}
$\int\limits_{0}^{\infty }e^{-zt}{\varphi _{x}\left( t\right) dt\text{, \ }%
\func{Re}z>0\text{;}}$ \\ 
${-}\int\limits_{-\infty }^{0}e^{-zt}{\varphi _{x}\left( t\right) dt\text{,
\ }\func{Re}z<0}$.%
\end{tabular}%
\right.
\end{equation*}%
This shows that the function $z\rightarrow \langle \varphi ,u\left( z\right)
\rangle $ is the Carleman transform of ${\varphi _{x}.}$ It follows that 
\begin{equation*}
i\text{sp}_{B}\left( \varphi _{x}\right) \subseteq \sigma _{T}\left(
x\right) \text{\ \ }\left( \forall \varphi \in X^{\ast }\right)
\end{equation*}%
and hence 
\begin{equation*}
\overline{\bigcup_{\varphi \in X^{\ast }}\text{sp}_{B}\left( \varphi
_{x}\right) }\subseteq -i\sigma _{T}\left( x\right) .
\end{equation*}%
To show the reverse inclusion, assume that $\lambda _{0}\in 
\mathbb{R}
$ and 
\begin{equation*}
\lambda _{0}\notin \overline{\bigcup_{\varphi \in X^{\ast }}\text{sp}%
_{B}\left( \varphi _{x}\right) }.
\end{equation*}%
Then there exist a neighborhood $U$ of $\overline{\bigcup_{\varphi \in
X^{\ast }}\text{sp}_{B}\left( \varphi _{x}\right) }$ and $\varepsilon >0$
such that 
\begin{equation*}
\left( \lambda _{0}-\varepsilon ,\lambda _{0}+\varepsilon \right) \cap 
\overline{U}=\emptyset .
\end{equation*}%
Since the algebra $L_{\omega }^{1}\left( 
\mathbb{R}
\right) $ is regular, there exists a function $f\in L_{\omega }^{1}\left( 
\mathbb{R}
\right) $ such that $\widehat{f}=1$ on $\left[ \lambda _{0}-\varepsilon
/2,\lambda _{0}+\varepsilon /2\right] $ and $\widehat{f}=0$ on $\overline{U}%
. $ Notice that $\widehat{f}$ vanishes in a neighborhood of sp$_{B}\left(
\varphi _{x}\right) $ and supp$\widehat{f}\subseteq \left[ \lambda
_{0}-\varepsilon ,\lambda _{0}+\varepsilon \right] .$ Consequently, $f$
belongs to the smallest ideal of $L_{\omega }^{1}\left( 
\mathbb{R}
\right) $ whose hull is sp$_{B}\left( \varphi _{x}\right) .$ Since 
\begin{equation*}
\text{sp}_{B}\left( \varphi _{x}\right) =\text{hull}\left\{ f\in L_{\omega
}^{1}\left( 
\mathbb{R}
\right) :f\ast \varphi _{x}=0\right\} ,
\end{equation*}%
we have that $\left( \lambda _{0}-\varepsilon /2,\lambda _{0}+\varepsilon
/2\right) \subset 
\mathbb{R}
\setminus $sp$_{B}\left( \varphi _{x}\right) $ for every $\varphi \in
X^{\ast }.$ It follows that the function $z\rightarrow \langle \varphi
,u\left( z\right) \rangle $ can be analytically extended to $i\left( \lambda
_{0}-\varepsilon /2,\lambda _{0}+\varepsilon /2\right) $ for every $\varphi
\in X^{\ast }.$ Hence, $u\left( z\right) $ can be analytically extended to $%
i\left( \lambda _{0}-\varepsilon /2,\lambda _{0}+\varepsilon /2\right) $ and
therefore $i\lambda _{0}\notin \sigma _{T}\left( x\right) $ or $\lambda
_{0}\notin -i\sigma _{T}\left( x\right) .$ Thus we have 
\begin{equation*}
\overline{\bigcup_{\varphi \in X^{\ast }}\text{sp}_{B}\left( \varphi
_{x}\right) }=-i\sigma _{T}\left( x\right) .
\end{equation*}%
Further, it is easy to see that 
\begin{equation*}
I_{x}=\bigcap\limits_{\varphi \in X^{\ast }}\mathcal{I}_{\widetilde{\varphi
_{x}}},
\end{equation*}%
where $\widetilde{\varphi _{x}}\left( t\right) =\varphi _{x}\left( -t\right)
.$ Taking into account that 
\begin{equation*}
\text{sp}_{B}\left( \widetilde{\varphi _{x}}\right) =\left\{ -\lambda
:\lambda \in \text{sp}_{B}\left( \varphi _{x}\right) \right\} ,
\end{equation*}%
by (2.2) we can write 
\begin{equation*}
\text{hull}\left( I_{x}\right) =\overline{\bigcup_{\varphi \in X^{\ast }}%
\text{hull}\left( \mathcal{I}_{\widetilde{\varphi _{x}}}\right) }=\overline{%
\bigcup_{\varphi \in X^{\ast }}\text{sp}_{B}\left( \widetilde{\varphi _{x}}%
\right) }=i\sigma _{T}\left( x\right) .
\end{equation*}

Thus we have the following:

\begin{proposition}
Let $\omega $ be a regular weight. Assume that $T\in B\left( X\right) $ has
SVEP and $x\in X$ satisfies the condition $\left\Vert e^{tT}x\right\Vert
\leq C_{x}\omega \left( t\right) $ for all $t\in 
\mathbb{R}
$ and for some constant $C_{x}>0.$ Then 
\begin{equation*}
i\sigma _{T}\left( x\right) =\mathnormal{hull}\left( I_{x}\right) .
\end{equation*}
\end{proposition}

\begin{corollary}
Assume that the hypotheses of Proposition 2.3 are satisfied. The following
assertions hold for every $f\in L_{\omega }^{1}\left( 
\mathbb{R}
\right) $:

$\left( a\right) $ $\sigma _{T}\left( x_{f}\right) \subseteq \sigma
_{T}\left( x\right) .$

$\left( b\right) $ $i\sigma _{T}\left( x_{f}\right) \subseteq \mathnormal{%
supp}\widehat{f}.$
\end{corollary}

\begin{proof}
(a) By Proposition 2.3, it suffices to show that $I_{x}\subseteq I_{x_{f}}.$
Indeed, if $g\in I_{x},$ then as $x_{g}=0$ we have 
\begin{equation*}
\left( x_{f}\right) _{g}=x_{f\ast g}=\left( x_{g}\right) _{f}=0.
\end{equation*}
This shows that $g\in I_{x_{f}}.$

(b) Assume that the Fourier transform of $g\in L_{\omega }^{1}\left( 
\mathbb{R}
\right) $ vanishes on supp$\widehat{f}.$ Then as $f\ast g=0,$ we have 
\begin{equation*}
\left( x_{f}\right) _{g}=x_{f\ast g}=0
\end{equation*}
and therefore, $g\in I_{x_{f}}.$ Consequently, $\widehat{g}$ vanishes on hull%
$\left( I_{x_{f}}\right) .$ By Proposition 2.3, $\widehat{g}$ vanishes on $%
i\sigma _{T}\left( x_{f}\right) .$
\end{proof}

For $f\in L_{\omega }^{1}\left( 
\mathbb{R}
\right) $ and $s\in 
\mathbb{R}
,$ let $f_{s}\left( t\right) :=f\left( t-s\right) .$ Let $e_{n}:=2n\chi _{%
\left[ -1/n,1/n\right] }$ $\left( n\in 
\mathbb{N}
\right) $, where $\chi _{\left[ -1/n,1/n\right] }$ is the characteristic
function of the interval $\left[ -1/n,1/n\right] .$ If $K:=\sup_{t\in \left[
-1,1\right] }\omega \left( t\right) ,$ then $\left\Vert e_{n}\right\Vert
_{\omega }\leq K$ for all $n\in 
\mathbb{N}
.$ On the other hand, by continuity of the mapping $s\mapsto f_{s}$, we have 
\begin{equation*}
\lim_{n\rightarrow \infty }\left\Vert f\ast e_{n}-f\right\Vert _{1,\omega
}=0.
\end{equation*}%
Consequently, $\left\{ e_{n}\right\} $ is a bounded approximate identity
(b.a.i.) for $L_{\omega }^{1}\left( 
\mathbb{R}
\right) .$ If $x\in E_{T}^{\omega },$ from the identity 
\begin{equation*}
x_{e_{n}}-x=\int_{-1/n}^{1/n}e_{n}\left( t\right) \left( e^{tT}x-x\right) dt,
\end{equation*}%
it follows that $x_{e_{n}}\rightarrow x.$ Similarly, $x_{f\ast
e_{n}}\rightarrow x_{f}$ for every $f\in L_{\omega }^{1}\left( 
\mathbb{R}
\right) .$

\begin{proposition}
Let $\omega $ be a regular weight and $x\in X.$ Assume that $T\in B\left(
X\right) $ has SVEP and $\left\Vert e^{tT}x\right\Vert \leq C_{x}\omega
\left( t\right) $ for all $t\in 
\mathbb{R}
$ and for some constant $C_{x}>0.$ For an arbitrary $f\in L_{\omega
}^{1}\left( 
\mathbb{R}
\right) ,$ the following assertions hold:

$\left( a\right) $ If $x_{f}=0,$ then $\widehat{f}$ vanishes on $i\sigma
_{T}\left( x\right) .$

$\left( b\right) $ If $\widehat{f}$ vanishes in a neighborhood of $i\sigma
_{T}\left( x\right) ,$ then $x_{f}=0$.

$\left( c\right) $ If $\widehat{f}=1$ in a neighborhood of $i\sigma
_{T}\left( x\right) ,$ then $x_{f}=x$.
\end{proposition}

\begin{proof}
(a) By Proposition 2.3, $i\sigma _{T}\left( x\right) =$hull$\left(
I_{x}\right) $ and therefore $I_{x}\subseteq I_{\omega }\left( i\sigma
_{T}\left( x\right) \right) .$ This clearly implies (a).

(b) Let $g\in L_{\omega }^{1}\left( 
\mathbb{R}
\right) $ be such that supp$\widehat{g}$ is compact. Then $f\ast g\in
J_{\omega }\left( i\sigma _{T}\left( x\right) \right) $ and therefore $f\ast
g\in I_{x}.$ So we have $x_{f\ast g}=0.$ Since the algebra $L_{\omega
}^{1}\left( 
\mathbb{R}
\right) $ is Tauberian, $x_{f\ast g}=0$ for all $g\in L_{\omega }^{1}\left( 
\mathbb{R}
\right) .$ It follows that $x_{f\ast e_{n}}=0$ for all $n,$ where $\left\{
e_{n}\right\} $ is a b.a.i. for $L_{\omega }^{1}\left( 
\mathbb{R}
\right) .$ As $n\rightarrow \infty ,$ we have $x_{f}=0.$

(c) Since the Fourier transform of $f\ast e_{n}-e_{n}$ vanishes in a
neighborhood of $i\sigma _{T}\left( x\right) ,$ by (b), $x_{f\ast
e_{n}}=x_{e_{n}}$. As $n\rightarrow \infty ,$ we have $x_{f}=x.$
\end{proof}

As an immediate consequence of Proposition 2.5, we have the following:

\begin{corollary}
Assume that the hypotheses of Proposition 2.5 are satisfied and $f\in
L_{\omega }^{1}\left( 
\mathbb{R}
\right) .$ If $i\sigma _{T}\left( x\right) $ is a set of synthesis for $%
L_{\omega }^{1}%
\mathbb{R}
,$ then:

$\left( a\right) $ $x_{f}=0$ if and only if $\widehat{f}$ vanishes on $%
i\sigma _{T}\left( x\right) .$

$\left( b\right) $ If $\widehat{f}=1$ on $i\sigma _{T}\left( x\right) ,$
then $x_{f}=x$.
\end{corollary}

By $S\left( 
\mathbb{R}
\right) $ we denote the set of all rapidly decreasing functions on $%
\mathbb{R}
,$ i.e. the set of all infinitely differentiable functions $\phi $ on $%
\mathbb{R}
$ such that 
\begin{equation*}
\lim_{\left\vert t\right\vert \rightarrow \infty }\left\vert t^{n}\phi
^{\left( k\right) }\left( t\right) \right\vert =0\text{, \ }\forall
n,k=0,1,2,....
\end{equation*}%
(here, $n$ can be replaced by any $\alpha \geq 0$)$.$ It can be seen that if 
$\omega $ is a polynomial weight, then $S\left( 
\mathbb{R}
\right) \subset L_{\omega }^{1}\left( 
\mathbb{R}
\right) .$

\begin{lemma}
Assume that $T\in B\left( X\right) $ and $x\in X$ satisfy the condition 
\begin{equation*}
\left\Vert e^{tT}x\right\Vert \leq C_{x}\left( 1+\left\vert t\right\vert
\right) ^{\alpha }\text{ }\left( \alpha \geq 0\right) ,
\end{equation*}
for all $t\in 
\mathbb{R}
$ and for some constant $C_{x}>0.$ Then for an arbitrary $\phi \in S\left( 
\mathbb{R}
\right) ,$%
\begin{equation*}
x_{\phi ^{\left( k\right) }}=\left( -1\right) ^{k}T^{k}x_{\phi }\text{, \ }%
\forall k\in 
\mathbb{N}
.
\end{equation*}
\end{lemma}

\begin{proof}
For an arbitrary $a,b\in 
\mathbb{R}
$ $\left( a<b\right) ,$ we can write 
\begin{equation*}
\int_{a}^{b}\phi ^{\prime }\left( t\right) e^{tT}xdt=\phi \left( b\right)
e^{bT}x-\phi \left( a\right) e^{aT}x-T\int_{a}^{b}\phi \left( t\right)
e^{tT}xdt.
\end{equation*}%
On the other hand, 
\begin{eqnarray*}
\left\Vert \phi \left( b\right) e^{bT}x-\phi \left( a\right)
e^{aT}x\right\Vert &\leq &C_{x}\left\vert \phi \left( b\right) \right\vert
\left( 1+\left\vert b\right\vert \right) ^{\alpha }+C_{x}\left\vert \phi
\left( a\right) \right\vert \left( 1+\left\vert a\right\vert \right)
^{\alpha } \\
&\leq &2^{\alpha }C_{x}\left( \left\vert \phi \left( b\right) \right\vert
+\left\vert \phi \left( b\right) \right\vert \left\vert b\right\vert
^{\alpha }+\left\vert \phi \left( a\right) \right\vert +\left\vert \phi
\left( a\right) \right\vert \left\vert a\right\vert ^{\alpha }\right) .
\end{eqnarray*}%
Since $\phi \in S\left( 
\mathbb{R}
\right) ,$ it follows that 
\begin{equation*}
\underset{a\rightarrow -\infty }{\lim_{b\rightarrow +\infty }}\left\Vert
\phi \left( b\right) e^{bT}x-\phi \left( a\right) e^{aT}x\right\Vert =0.
\end{equation*}%
Hence $x_{\phi ^{\prime }}=-Tx_{\phi }.$ By induction we obtain our result.
\end{proof}

We are now able to prove Theorem 2.1.

\begin{proof}[Proof of Theorem 2.1]
(a) follows from Proposition 2.2. To prove (b), observe that if $\sigma
_{T}\left( x\right) =\left\{ i\lambda \right\} $ $\left( \lambda \in 
\mathbb{R}
\right) ,$ then $\sigma _{T-i\lambda I}\left( x\right) =\left\{ 0\right\} .$
Moreover, the operator $T-i\lambda I$ has SVEP and%
\begin{equation*}
\left\Vert e^{t\left( T-i\lambda I\right) }x\right\Vert =\left\Vert
e^{tT}x\right\Vert \leq C_{x}\left( 1+\left\vert t\right\vert \right)
^{\alpha }.
\end{equation*}%
Therefore, it suffices to show that if $\sigma _{T}\left( x\right) =\left\{
0\right\} ,$ then 
\begin{equation*}
x_{f}=\widehat{f}\left( 0\right) x+\frac{\widehat{f}^{\prime }\left(
0\right) }{1!}\left( iT\right) x+...+\frac{\widehat{f}^{\left( k\right)
}\left( 0\right) }{k!}\left( iT\right) ^{k}x,
\end{equation*}%
for all $f\in L_{\omega }^{1}\left( 
\mathbb{R}
\right) .$ As we have noted above, if $f\in L_{\omega }^{1}\left( 
\mathbb{R}
\right) $ then the first $k$ derivatives of the Fourier transform of $f$
exist at each point of $s\in 
\mathbb{R}
$ and 
\begin{equation*}
J_{\omega }\left( \left\{ s\right\} \right) =\left\{ f\in L_{\omega
}^{1}\left( 
\mathbb{R}
\right) :\widehat{f}\left( s\right) =\widehat{f}^{\prime }\left( s\right)
=...=\widehat{f}^{\left( k\right) }\left( s\right) =0\right\} ,
\end{equation*}%
where $k=\left[ \alpha \right] $. Recall that $J_{\omega }\left( \left\{
s\right\} \right) $ is the smallest closed ideal of $L_{\omega }^{1}\left( 
\mathbb{R}
\right) $ whose hull is $\left\{ s\right\} .$ By Proposition 2.3, hull$%
\left( I_{x}\right) =\left\{ 0\right\} $ and therefore, $J_{\omega }\left(
\left\{ 0\right\} \right) \subseteq I_{x}$. Now, let $\phi \in S\left( 
\mathbb{R}
\right) $ be such that $\widehat{\phi }=1$ in a neighborhood of $0.$ For a
given $f\in L_{\omega }^{1}\left( 
\mathbb{R}
\right) ,$ consider the function 
\begin{equation*}
g:=f-\widehat{f}\left( 0\right) \phi -\frac{\widehat{f}^{\prime }\left(
0\right) }{1!}\left( -i\right) \phi ^{\prime }-...-\frac{\widehat{f}^{\left(
k\right) }\left( 0\right) }{k!}\left( -i\right) ^{k}\phi ^{\left( k\right) }.
\end{equation*}%
As%
\begin{equation*}
\widehat{\phi ^{\left( k\right) }}\left( s\right) =\left( is\right) ^{k}%
\widehat{\phi }\left( s\right) ,
\end{equation*}%
we have%
\begin{equation*}
\widehat{g}\left( s\right) =\widehat{f}\left( s\right) -\left[ \widehat{f}%
\left( 0\right) +\frac{\widehat{f}^{\prime }\left( 0\right) }{1!}s+...+\frac{%
\widehat{f}^{\left( k\right) }\left( 0\right) }{k!}s^{k}\right] \widehat{%
\phi }\left( s\right) .
\end{equation*}%
It can be seen that the first $k$ derivatives of the function $\widehat{g}$
at $0$ are zero and therefore, $g\in J_{\omega }\left( \left\{ 0\right\}
\right) .$ Consequently, $g\in I_{x}$ and so $x_{g}=0.$ On the other hand,
by Lemma 2.7, 
\begin{equation*}
x_{\phi ^{\left( k\right) }}=\left( -1\right) ^{k}T^{k}x
\end{equation*}%
which implies%
\begin{equation*}
x_{g}=x_{f}-\widehat{f}\left( 0\right) x-\frac{\widehat{f}^{\prime }\left(
0\right) }{1!}\left( iT\right) x-...-\frac{\widehat{f}^{\left( k\right)
}\left( 0\right) }{k!}\left( iT\right) ^{k}x.
\end{equation*}%
So we have 
\begin{equation*}
x_{f}=\widehat{f}\left( 0\right) x+\frac{\widehat{f}^{\prime }\left(
0\right) }{1!}\left( iT\right) x+...+\frac{\widehat{f}^{\left( k\right)
}\left( 0\right) }{k!}\left( iT\right) ^{k}x.
\end{equation*}
\end{proof}

Assume that $T\in B\left( X\right) $ has SVEP, $\left( T-\lambda I\right)
^{k}x=0,$ and $\left( T-\lambda I\right) ^{k-1}x\neq 0$ for some $x\in X,$ $%
\lambda \in 
\mathbb{C}
$, and $k\in 
\mathbb{N}
.$ We claim that $\sigma _{T}\left( x\right) =\left\{ \lambda \right\} .$
Indeed, it is easy to check that if 
\begin{equation*}
u\left( z\right) :=\frac{x}{z-\lambda }+\frac{\left( T-\lambda I\right) x}{%
\left( z-\lambda \right) ^{2}}+...+\frac{\left( T-\lambda I\right) ^{k-1}x}{%
\left( z-\lambda \right) ^{k}},
\end{equation*}%
then $\left( zI-T\right) u\left( z\right) =x$ for all $z\in 
\mathbb{C}
\diagdown \left\{ \lambda \right\} .$ This shows that $\sigma _{T}\left(
x\right) \subseteq \left\{ \lambda \right\} .$ Since $T$ has SVEP and $%
\sigma _{T}\left( x\right) \neq \emptyset $, we have $\sigma _{T}\left(
x\right) =\left\{ \lambda \right\} .$

As a consequence of Theorem 2.1, we have the following:

\begin{corollary}
Assume that the hypotheses of Theorem 2.1 are satisfied. If $\sigma
_{T}\left( x\right) =\left\{ i\lambda \right\} $ $\left( \lambda \in 
\mathbb{R}
\right) ,$ then $\left( T-i\lambda I\right) ^{[\alpha ]+1}x=0.$
\end{corollary}

\begin{proof}
As in the proof of Theorem 2.1, we may assume that $\sigma _{T}\left(
x\right) =\left\{ 0\right\} .$ Let us show that $T^{k+1}x=0,$ where $k=\left[
\alpha \right] .$ Let $\phi \in S\left( 
\mathbb{R}
\right) $ be such that $\widehat{\phi }=1$ in a neighborhood of $0$ and $%
f:=\phi ^{\left( k+1\right) }.$ Since 
\begin{equation*}
\widehat{f}\left( 0\right) =\widehat{f}^{\prime }\left( 0\right) =...=%
\widehat{f}^{\left( k\right) }\left( 0\right) =0,
\end{equation*}%
by Theorem 2.1,%
\begin{equation*}
0=x_{f}=\left( -1\right) ^{k+1}T^{k+1}x.
\end{equation*}
\end{proof}

\begin{corollary}
Assume that the hypotheses of Theorem 2.1 are satisfied and $0\leq \alpha
<1. $ If $\sigma _{T}\left( x\right) =\left\{ i\lambda _{1},...,i\lambda
_{N}\right\} $, where $\lambda _{1},...,\lambda _{N}$ are pairwice distinct
real numbers$,$ then there exist $x_{1},...,x_{N}$ in $X$ such that $%
x=x_{1}+...+x_{N}$ and $Tx_{k}=i\lambda _{k}x_{k}$ $(k=1,...,N).$
\end{corollary}

\begin{proof}
Let $f\in L_{\omega }^{1}\left( 
\mathbb{R}
\right) $ be such that 
\begin{equation*}
\widehat{f}\left( -\lambda _{1}\right) =...=\widehat{f}\left( -\lambda
_{N}\right) =1.
\end{equation*}%
Since $\left\{ -\lambda _{1},...,-\lambda _{N}\right\} $ is a set of
synthesis for $L_{\omega }^{1}\left( 
\mathbb{R}
\right) ,$ by Corollary 2.6, $x_{f}=x.$ Now, assume that the functions $%
f_{1},...,f_{N}$ in $L_{\omega }^{1}\left( 
\mathbb{R}
\right) $ satisfy the conditions 
\begin{equation*}
\widehat{f_{m}}\left( -\lambda _{n}\right) =\delta _{mn}\text{ \ }\left(
m,n=1,...,N\right) ,
\end{equation*}%
where $\delta _{mn}$ is the Kronecker symbol. We put $x_{k}=x_{f_{k}}.$
Since the Fourier transform of the function $f-f_{1}-...-f_{N}$ vanishes on $%
\left\{ -\lambda _{1},...,-\lambda _{N}\right\} ,$ by Corollary 2.6,%
\begin{equation*}
x=x_{f}=x_{1}+...+x_{N}.
\end{equation*}%
On the other hand, by Corollary 2.4, $\sigma _{T}\left( x_{k}\right)
\subseteq \left\{ i\lambda _{k}\right\} .$ Since $T$ has SVEP, $\sigma
_{T}\left( x_{k}\right) =\left\{ i\lambda _{k}\right\} .$ By Corollary 2.8, $%
Tx_{k}=i\lambda _{k}x_{k}$ $(k=1,...,N).$
\end{proof}

\begin{example}
The condition $\left\Vert e^{tT}x\right\Vert \leq C_{x}\left( 1+t\right)
^{\alpha }$ for all $t\geq 0$ is not sufficient in Corollary 2.8. To see
this, let $S$ be a contraction on a Hilbert space $H$ such that $\sigma
\left( S\right) =\left\{ 1\right\} $ and $S\neq I$ $($if $S$ is a $C_{0}-$%
contraction on $H$ with minimal function $\exp \frac{z+1}{z-1}$, then $%
\sigma \left( S\right) =\left\{ 1\right\} $ and $S\neq I$ \cite[Ch.III]{16}$%
) $. If $T:=S-I,$ then $T\neq 0,$ $\sigma \left( T\right) =0,$ and $%
\left\Vert e^{tT}x\right\Vert \leq \left\Vert x\right\Vert $ for all $x\in H$
and $t\geq 0$. Let $x_{0}\in H$ be such that $Tx_{0}\neq 0.$ Since $%
x_{0}\neq 0$ and $T$ has SVEP, $\sigma _{T}\left( x_{0}\right) \neq
\emptyset .$ On the other hand, as $\sigma _{T}\left( x_{0}\right) \subseteq
\sigma \left( T\right) =\left\{ 0\right\} ,$ we have $\sigma _{T}\left(
x_{0}\right) =\left\{ 0\right\} .$
\end{example}

Assume that the condition (1.1) is satisfied for all $x\in X.$ Then, by the
uniform boundedness principle there exists a constant $C_{T}>0$ such that 
\begin{equation}
\left\Vert e^{tT}\right\Vert \leq C_{T}\left( 1+\left\vert t\right\vert
\right) ^{\alpha },\text{ \ }\forall t\in 
\mathbb{R}
.  \label{2.5}
\end{equation}%
Let $T$ be the backward shift operator on the Hardy space $H^{2}$ defined by 
\begin{equation*}
\left( Tf\right) \left( z\right) =\frac{f\left( z\right) -f\left( 0\right) }{%
z}.
\end{equation*}%
From this it is readily verified that 
\begin{equation*}
\left\Vert e^{tT}z^{k}\right\Vert _{2}\leq C_{k}\left( 1+\left\vert
t\right\vert \right) ^{k},\text{ \ }\forall k\in 
\mathbb{N}
\cup \left\{ 0\right\} .
\end{equation*}%
It follows that the operator $T$ does not satisfy the condition (2.5).

We recall that according to \cite[Definition 1.1.1]{12}, an operator $T\in
B\left( X\right) $ is \textit{decomposable} if for every cover of the
complex plane by a pair $\left\{ U,W\right\} $ of open sets, there exist a
pair of closed $T-$invariant subspaces $E$ and $F$ of $X$ such that $E+F=X$, 
$\sigma \left( T\mid _{E}\right) \subset U$, and $\sigma \left( T\mid
_{F}\right) \subset W$.

Every decomposable operator $T$ on a Banach space $X$ has SVEP \cite[Ch.1]%
{12}. Moreover, $\lim_{n\rightarrow \infty }\left\Vert T^{n}x\right\Vert ^{%
\frac{1}{n}}$ exists for all $x\in X$ \cite[Proposition 3.3.17]{12}.

For an arbitrary $k\in \left\{ 0,1,...,\infty \right\} ,$ let $C^{k}\left(
\Omega \right) $ be the algebra of all $k-$times continuously differentiable
complex-valued functions defined on an open subset $\Omega $ of the complex
plane. An operator $A\in B\left( X\right) $ is said to be $C^{k}\left(
\Omega \right) -$\textit{scalar operator} if there exists a continuous
algebra homomorphism\textit{\ }$\Phi :C^{k}\left( \Omega \right) \rightarrow
B\left( X\right) $ for which $\Phi \left( \mathbf{1}\right) =I$ and $\Phi
\left( \mathbf{z}\right) =A$ \cite{5,12}. In the case $\Omega =%
\mathbb{C}
$ and $k=\infty ,$ the operator $A$ is said to be \textit{generalized scalar 
}(in this case, $\Phi $ is continuous with respect to the natural Fr\'{e}%
chet algebra topology on $C^{\infty }\left( 
\mathbb{C}
\right) )$. It is known \cite[Theorem 1.4.10]{12} that every $C^{k}\left(
\Omega \right) -$scalar operator is decomposable.

Every generalized scalar operator $A$ has a decomposition $A=T+iS$ with $T,S$
commuting generalized scalar operators with real spectra \cite[Lemma 4.6.1]%
{5}. Recall also \cite[Theorem 1.5.19]{12} that an operator $T$ with real
spectrum is generalized scalar if and only if there exist constants $C_{T}>0$
and $\alpha \geq 0$ such that 
\begin{equation}
\left\Vert e^{itT}\right\Vert \leq C_{T}\left( 1+\left\vert t\right\vert
^{\alpha }\right) ,\text{ }\forall t\in 
\mathbb{R}
.  \label{2.6}
\end{equation}

An operator $T\in B\left( X\right) $ is said to be \textit{Hermitian }if $%
\left\Vert e^{itT}\right\Vert =1$ for all $t\in 
\mathbb{R}
$. A Hermitian operator is $C^{1}\left( 
\mathbb{R}
\right) -$scalar and therefore it is decomposable \cite{2}. A basis result
for Hermitian operators is that $\left\Vert T\right\Vert =r\left( T\right) $ 
\cite[p.57]{3}.

We have the following:

\begin{theorem}
Assume that $T\in B\left( X\right) $ has SVEP and $x\in X$ satisfies the
condition 
\begin{equation*}
\left\Vert e^{tT}x\right\Vert \leq C_{x}\left( 1+\left\vert t\right\vert
^{\alpha }\right) \text{ }\left( 0\leq \alpha <1\right) ,
\end{equation*}
for all $t\in 
\mathbb{R}
$ and for some constant $C_{x}>0.$ Then we have 
\begin{equation*}
\left\Vert Tx\right\Vert \leq C_{x}\left[ r_{T}\left( x\right) +C\left(
\alpha \right) r_{T}\left( x\right) ^{1-\alpha }\right] ,
\end{equation*}%
where%
\begin{equation}
C\left( \alpha \right) =\left( \frac{2}{\pi }\right) ^{2-\alpha
}\sum\limits_{k\in 
\mathbb{Z}
}\frac{1}{\left\vert 2k+1\right\vert ^{2-\alpha }}\text{ \ }\left( C\left(
0\right) =1\right) .  \label{2.7}
\end{equation}
\end{theorem}

The following lemma is needed for the proof of the theorem.

\begin{lemma}
Let $\mu \in M_{\omega }\left( 
\mathbb{R}
\right) ,$ where $\omega \left( t\right) =\left( 1+\left\vert t\right\vert
\right) ^{\alpha }$ $\left( \alpha \geq 0\right) .$ Assume that $T\in
B\left( X\right) $ and $x\in X$ satisfy the following conditions:

$\left( i\right) $ $T$ has SVEP;

$\left( ii\right) $ $\left\Vert e^{tT}x\right\Vert \leq C_{x}\omega \left(
t\right) $ for all $t\in 
\mathbb{R}
$ and for some constant $C_{x}>0.$

If $\widehat{\mu }\left( \lambda \right) =\lambda $ in a neighborhood of $%
i\sigma _{T}\left( x\right) ,$ then 
\begin{equation*}
x_{\mu }=iTx.
\end{equation*}
\end{lemma}

\begin{proof}
Let $g\in S\left( 
\mathbb{R}
\right) $ be such that $\widehat{g}=1$ in a neighborhood of $i\sigma
_{T}\left( x\right) .$ By Proposition 2.5, $x_{g}=x.$ On the other hand, by
Lemma 2.7, 
\begin{equation*}
x_{g^{\prime }}=-Tx.
\end{equation*}%
Since%
\begin{equation*}
\widehat{g^{\prime }}\left( \lambda \right) =i\lambda \widehat{g}\left(
\lambda \right) ,
\end{equation*}%
the Fourier transform of the function $-ig^{\prime }-\mu \ast g$ vanishes in
a neighborhood of $i\sigma _{T}\left( x\right) .$ By Proposition 2.5, 
\begin{equation*}
-ix_{g^{\prime }}=x_{\mu \ast g}=\left( x_{g}\right) _{\mu }=x_{\mu }.
\end{equation*}%
Hence $x_{\mu }=iTx.$
\end{proof}

Note that in the preceding lemma the weight function $\omega \left( t\right)
=\left( 1+\left\vert t\right\vert \right) ^{\alpha }$ $\left( \alpha \geq
0\right) $ can be replaced by the weight $\omega \left( t\right)
=1+\left\vert t\right\vert ^{\alpha }$ $\left( 0\leq \alpha <1\right) .$

\begin{proof}[Proof of Theorem 2.11]
We basically follow the proof in \cite[Lemma 3.4]{14}. Let $a>0$ be such
that $i\sigma _{T}\left( x\right) \subset \left( -a,a\right) .$ Consider the
function $f$ on $%
\mathbb{R}
$ defined by $f\left( \lambda \right) =\lambda $ for $-a\leq \lambda \leq a$
and $f\left( \lambda \right) =2a-\lambda $ for $a\leq \lambda \leq 3a.$ We
extend this function periodically to the real line by putting $f\left(
\lambda +4a\right) =f\left( \lambda \right) $ $\left( \lambda \in 
\mathbb{R}
\right) .$ The Fourier coefficients of $f$ are given by the equalities%
\begin{equation*}
c_{n}\left( f\right) =\frac{1}{4a}\int_{-a}^{3a}\exp \left( -i\frac{n\pi }{2a%
}\lambda \right) f\left( \lambda \right) d\lambda .
\end{equation*}%
A few lines of computation show that 
\begin{equation*}
c_{2k}\left( f\right) =0\text{ \ and \ }c_{2k+1}\left( f\right) =\frac{1}{i}%
\frac{4a}{\pi ^{2}}\left( -1\right) ^{k}\frac{1}{\left( 2k+1\right) ^{2}}%
\text{ \ }\left( k\in 
\mathbb{Z}
\right) .
\end{equation*}%
Let $\mu $ be a discrete measure on $%
\mathbb{R}
$ concentrated at the points 
\begin{equation*}
\lambda _{k}:=-\frac{1}{a}\left( 2k+1\right) \frac{\pi }{2}\text{ }
\end{equation*}%
with the corresponding weights 
\begin{equation*}
c_{k}:=\frac{1}{i}\frac{4a}{\pi ^{2}}\left( -1\right) ^{k}\frac{1}{\left(
2k+1\right) ^{2}}\text{ \ }\left( k\in 
\mathbb{Z}
\right) .
\end{equation*}%
Since 
\begin{equation*}
\sum\limits_{k\in 
\mathbb{Z}
}\left\vert c_{k}\right\vert <\infty ,
\end{equation*}%
by the uniqueness theorem,%
\begin{equation*}
\widehat{\mu }\left( \lambda \right) =f\left( \lambda \right) =\frac{1}{i}%
\frac{4a}{\pi ^{2}}\sum\limits_{k\in 
\mathbb{Z}
}\left( -1\right) ^{k}\frac{1}{\left( 2k+1\right) ^{2}}\exp \left[ i\frac{1}{%
a}\left( 2k+1\right) \frac{\pi }{2}\lambda \right] .
\end{equation*}%
Now if $\omega \left( t\right) :=1+\left\vert t\right\vert ^{\alpha }$ $%
\left( 0\leq \alpha <1\right) ,$ then as 
\begin{equation*}
\sum\limits_{k\in 
\mathbb{Z}
}\frac{1}{\left( 2k+1\right) ^{2}}=\frac{\pi ^{2}}{4},
\end{equation*}%
we can write

\begin{eqnarray*}
\left\Vert \mu \right\Vert _{1,\omega } &=&\int\limits_{%
\mathbb{R}
}\left( 1+\left\vert t\right\vert ^{\alpha }\right) d\left\vert \mu
\right\vert \left( t\right) \\
&=&\sum\limits_{k\in 
\mathbb{Z}
}\left\vert c_{k}\right\vert \left( 1+\left\vert \lambda _{k}\right\vert
^{\alpha }\right) \\
&=&\frac{4a}{\pi ^{2}}\sum\limits_{k\in 
\mathbb{Z}
}\frac{[1+\frac{1}{a^{\alpha }}\left\vert 2k+1\right\vert ^{\alpha }\left( 
\frac{\pi }{2}\right) ^{\alpha }]}{\left( 2k+1\right) ^{2}} \\
&=&a+C\left( \alpha \right) a^{1-\alpha },
\end{eqnarray*}%
where 
\begin{equation*}
C\left( \alpha \right) =\left( \frac{2}{\pi }\right) ^{2-\alpha
}\sum\limits_{k\in 
\mathbb{Z}
}\frac{1}{\left\vert 2k+1\right\vert ^{2-\alpha }}.
\end{equation*}%
Since $\widehat{\mu }\left( \lambda \right) =\lambda $ in a neighborhood of $%
i\sigma _{T}\left( x\right) ,$ by Lemma 2.12, $x_{\mu }=iTx$. Consequently,
we get 
\begin{equation*}
\left\Vert Tx\right\Vert =\left\Vert x_{\mu }\right\Vert \leq
C_{x}\left\Vert \mu \right\Vert _{1,\omega }=C_{x}\left[ a+C\left( \alpha
\right) a^{1-\alpha }\right] .
\end{equation*}%
Since $a$ is an arbitrary positive number greater than $r_{T}\left( x\right) 
$, we have 
\begin{equation*}
\left\Vert Tx\right\Vert \leq C_{x}\left[ r_{T}\left( x\right) +C\left(
\alpha \right) r_{T}\left( x\right) ^{1-\alpha }\right] .
\end{equation*}
\end{proof}

An obvious corollary of Theorem 2.11 is the next result.

\begin{corollary}
If the operator $T\in B\left( X\right) $ satisfies the condition $\left(
2.6\right) $ with $0\leq \alpha <1,$ then for every $x\in X,$ 
\begin{equation*}
\left\Vert Tx\right\Vert \leq C_{T}\left[ r_{T}\left( x\right) +C\left(
\alpha \right) r_{T}\left( x\right) ^{1-\alpha }\right] \left\Vert
x\right\Vert .
\end{equation*}%
In particular, we have%
\begin{equation*}
\left\Vert T\right\Vert \leq C_{T}\left[ r\left( T\right) +C\left( \alpha
\right) r\left( T\right) ^{1-\alpha }\right] .
\end{equation*}%
where $C\left( \alpha \right) $ is defined by $\left( 2.7\right) .$
\end{corollary}

An operator $A\in B\left( X\right) $ is called \textit{normal} if $A=T+iS$
with $T$, $S$ are commuting Hermitian operators on $X.$ A normal operator on
a Banach space is $C^{2}\left( 
\mathbb{C}
\right) -$scalar and therefore it is decomposable \cite{2}. If $A$ is a
normal operator on a Banach space, then $\left\Vert A\right\Vert \leq
2r\left( A\right) $ (see, Lemma 2.15). In \cite[Corollary 3]{2}, it was
proved that if $A=T+iS$ is a normal operator on a Banach space $X$ and $%
r_{A}\left( x\right) =0$ for some $x\in X$, then $Tx=Sx=0.$ More general
result was obtained in \cite[Corollary 4]{4}: 
\begin{equation*}
\max \left\{ \left\Vert Tx\right\Vert ,\left\Vert Sx\right\Vert \right\}
\leq r_{A}\left( x\right) \left\Vert x\right\Vert ,\text{ \ }\forall x\in X.
\end{equation*}

We have the following:

\begin{theorem}
Let $T$ and $S$ are two commuting operators in $B\left( X\right) $
satisfying the condition $(2.6)$ with constants $\left\{ C_{T},\alpha
\right\} $ and $\left\{ C_{S},\alpha \right\} ,$ respectively. If $0\leq
\alpha <1,$ then for every $x\in X,$ 
\begin{equation*}
\max \left\{ \left\Vert Tx\right\Vert ,\left\Vert Sx\right\Vert \right\}
\leq \max \left\{ C_{T},C_{S}\right\} \left[ r_{T+iS}\left( x\right)
+C\left( \alpha \right) r_{T+iS}\left( x\right) ^{1-\alpha }\right]
\left\Vert x\right\Vert ,
\end{equation*}%
where $C\left( \alpha \right) $ is defined by $\left( 2.7\right) .$
\end{theorem}

For the proof, we need some preliminary results.

\begin{lemma}
If $T$ and $S$ are two commuting operators in $B\left( X\right) $ with real
spectra, then 
\begin{equation*}
\max \left\{ r\left( T\right) ,r\left( S\right) \right\} \leq r\left(
T+iS\right) .
\end{equation*}

\begin{proof}
Let $\mathcal{A}$ be the maximal commutative subalgebra of $B\left( X\right) 
$ containing $T$ and $S.$ Clearly, $\mathcal{A}$ is unital. Denote by $%
\sigma _{\mathcal{A}}\left( A\right) $ the spectrum of $A\in \mathcal{A}$
with respect to the algebra $\mathcal{A}$. Let $\Sigma _{\mathcal{A}}$ be
the maximal ideal space of $\mathcal{A}$. Since $\mathcal{A}$ is a full
subalgebra of $B\left( X\right) ,$ we have 
\begin{eqnarray*}
\sigma \left( T+iS\right) &=&\sigma _{\mathcal{A}}\left( T+iS\right) \\
&=&\left\{ \phi \left( T+iS\right) :\phi \in \Sigma _{\mathcal{A}}\right\} \\
&=&\left\{ \phi \left( T\right) +i\phi \left( S\right) :\phi \in \Sigma _{%
\mathcal{A}}\right\}
\end{eqnarray*}%
and therefore, 
\begin{equation*}
r\left( T+iS\right) =\sup_{\phi \in \Sigma _{\mathcal{A}}}\left\vert \phi
\left( T\right) +i\phi \left( S\right) \right\vert .
\end{equation*}%
Similarly, as%
\begin{equation*}
\sigma \left( T\right) =\left\{ \phi \left( T\right) :\phi \in \Sigma _{%
\mathcal{A}}\right\} \text{ and }\sigma \left( S\right) =\left\{ \phi \left(
S\right) :\phi \in \Sigma _{\mathcal{A}}\right\} ,
\end{equation*}%
we have%
\begin{equation*}
r\left( T\right) =\sup_{\phi \in \Sigma _{\mathcal{A}}}\left\vert \phi
\left( T\right) \right\vert \text{ and }r\left( S\right) =\sup_{\phi \in
\Sigma _{\mathcal{A}}}\left\vert \phi \left( S\right) \right\vert .
\end{equation*}%
Since $\phi \left( T\right) $ and $\phi \left( S\right) $ are reals, we
obtain that 
\begin{equation*}
r\left( T\right) \leq r\left( T+iS\right) \text{ and }r\left( S\right) \leq
r\left( T+iS\right) .
\end{equation*}
\end{proof}
\end{lemma}

For a closed subset $F$ of $%
\mathbb{C}
,$ the \textit{local} \textit{spectral subspace} $X_{T}\left( F\right) $ of $%
T\in B\left( X\right) $ is defined by 
\begin{equation*}
X_{T}\left( F\right) =\left\{ x\in X:\sigma _{T}\left( x\right) \subseteq
F\right\} .
\end{equation*}%
An operator $T\in B\left( X\right) $ has \textit{Dunford's property }$\left(
C\right) $ if $X_{T}\left( F\right) $ is closed for every closed set $%
F\subset 
\mathbb{C}
.$ Every decomposable operator has Dunford's property\textit{\ }$\left(
C\right) $ \cite[Theorem 1.2.23]{12}. If $T$ has Dunford's property\textit{\ 
}$\left( C\right) ,$ then $\sigma \left( T\right) =\sigma _{T}\left(
x\right) $ holds for every cyclic vector $x$ for $T$ \cite[p.238]{12}.

\begin{lemma}
Let $T$ and $S$ are two commuting operators in $B\left( X\right) $
satisfying the condition $(2.6)$ with constants $\left\{ C_{T},\alpha
\right\} $ and $\left\{ C_{S},\beta \right\} ,$ respectively. Then, for
every $x\in X$ one has 
\begin{equation*}
\max \left\{ r_{T}\left( x\right) ,r_{S}\left( x\right) \right\} \leq
r_{T+iS}\left( x\right) .
\end{equation*}
\end{lemma}

\begin{proof}
Note that the operator $T+iS$ is $C^{k}\left( 
\mathbb{C}
\right) -$scalar with $k>\alpha +\beta +1$ \cite[Satz 3.3 and Lemma 3.5]{1}.
Consequently, $T+iS$ is decomposable and therefore it has Dunford's property%
\textit{\ }$\left( C\right) .$ For a given $x\in X,$ let $Y$ denote the
closed linear span of 
\begin{equation*}
\left\{ T^{n}S^{m}x:n,m\in 
\mathbb{N}
\cup \left\{ 0\right\} \right\} .
\end{equation*}%
Then the space $Y$ is invariant under $T+iS.$ Consequently, the operator $%
T\mid _{Y}+iS\mid _{Y}$ also has Dunford's property\textit{\ }$\left(
C\right) $ \cite[Proposition 1.2.21]{12}. Since $x$ is a cyclic vector of $%
T\mid _{Y}+iS\mid _{Y}$, we have 
\begin{equation*}
r\left( T\mid _{Y}+iS\mid _{Y}\right) =r_{T+iS}\left( x\right) .
\end{equation*}%
Observe that the operators $T\mid _{Y}$ and $S\mid _{Y}$ also satisfy the
condition (2.6) with constants ($C_{T},\alpha )$ and ($C_{S},\beta )$,
respectively. Therefore, the spectra of the operators $T\mid _{Y}$ and $%
S\mid _{Y}$ are reals. By Lemma 2.15, 
\begin{equation*}
\max \left\{ r\left( T\mid _{Y}\right) ,r\left( S\mid _{Y}\right) \right\}
\leq r\left( T\mid _{Y}+iS\mid _{Y}\right) =r_{T+iS}\left( x\right) .
\end{equation*}%
Since $\sigma _{T}\left( x\right) \subseteq \sigma _{T\mid _{Y}}\left(
x\right) $, we get%
\begin{equation*}
r_{T}\left( x\right) \leq r_{T\mid _{Y}}\left( x\right) \leq r\left( T\mid
_{Y}\right) .
\end{equation*}%
Similarly, $r_{S}\left( x\right) \leq r\left( S\mid _{Y}\right) .$ So we
have 
\begin{equation*}
\max \left\{ r_{T}\left( x\right) ,r_{S}\left( x\right) \right\} \leq \max
\left\{ r\left( T\mid _{Y}\right) ,r\left( S\mid _{Y}\right) \right\} \leq
r_{T+iS}\left( x\right) .
\end{equation*}
\end{proof}

Now we can prove Theorem 2.14.

\begin{proof}[Proof of Theorem 2.14]
For an arbitrary $x\in X,$ by Corollary 2.13 and Lemma 2.16, we can write%
\begin{eqnarray*}
\left\Vert Tx\right\Vert &\leq &C_{T}\left[ r_{T}\left( x\right) +C\left(
\alpha \right) r_{T}\left( x\right) ^{1-\alpha }\right] \left\Vert
x\right\Vert \\
&\leq &C_{T}\left[ r_{T+iS}\left( x\right) +C\left( \alpha \right)
r_{T+iS}\left( x\right) ^{1-\alpha }\right] \left\Vert x\right\Vert .
\end{eqnarray*}%
Similarly,%
\begin{equation*}
\left\Vert Sx\right\Vert \leq C_{S}\left[ r_{T+iS}\left( x\right) +C\left(
\alpha \right) r_{T+iS}\left( x\right) ^{1-\alpha }\right] \left\Vert
x\right\Vert .
\end{equation*}%
So we have 
\begin{equation*}
\max \left\{ \left\Vert Tx\right\Vert ,\left\Vert Sx\right\Vert \right\}
\leq \max \left\{ C_{T},C_{S}\right\} \left[ r_{T+iS}\left( x\right)
+C\left( \alpha \right) r_{T+iS}\left( x\right) ^{1-\alpha }\right]
\left\Vert x\right\Vert .
\end{equation*}
\end{proof}

\section{Applications}

For an invertible operator $A$ on a Hilbert space $H$, Deddens \cite{6}
introduced the set 
\begin{equation*}
\mathcal{B}_{A}:=\left\{ T\in B\left( H\right) :\sup_{n\in 
\mathbb{N}
}\left\Vert A^{n}TA^{-n}\right\Vert <\infty \right\} .
\end{equation*}%
Notice that $\mathcal{B}_{A}$ is a normed (not necessarily closed) algebra
with identity and contains the commutant $\left\{ A\right\} ^{\prime }$ of $%
A.$ In the same paper, Deddens shows that if $A\geq 0,$ then $\mathcal{B}%
_{A} $ coincides with the nest algebra associated with the nest of spectral
subspaces of $A.$ This gives a new and convenient characterization of nest
algebras. In \cite{6}, Deddens conjectured that the equality $\mathcal{B}%
_{A}=\left\{ A\right\} ^{\prime }$ holds if the spectrum of $A$ is reduced
to $\left\{ 1\right\} .$ In \cite{18}, Roth gave a negative answer to
Deddens conjecture. Williams \cite{19} proved that if $A\in B\left( X\right) 
$ is a quasinilpotent and $T\in B\left( X\right) $ satisfies the condition $%
\sup_{t\in 
\mathbb{R}
}\left\Vert e^{tA}Te^{-tA}\right\Vert <\infty ,$ then $AT=TA$ (for related
results see also \cite{15}).

For a fixed $A\in B\left( X\right) $ and $\alpha \geq 0,$ we consider the
class $\mathcal{D}_{A}^{\alpha }\left( 
\mathbb{R}
\right) $ of all operators $T\in B\left( X\right) $ for which there exists a
constant $C_{T}>0$ such that 
\begin{equation*}
\left\Vert e^{tA}Te^{-tA}\right\Vert \leq C_{T}\left( 1+\left\vert
t\right\vert \right) ^{\alpha },\text{ \ }\forall t\in 
\mathbb{R}
.
\end{equation*}%
Notice that $\mathcal{D}_{A}^{\alpha }\left( 
\mathbb{R}
\right) $ is a linear (not necessarily closed) subspace of $B\left( X\right) 
$ and contains the commutant of $A.$

As an applications of the results of the preceding section, here we give
complete characterization of the class $\mathcal{D}_{A}^{\alpha }\left( 
\mathbb{R}
\right) $ in the case when the spectrum of $A$ consists of one-point.

Let $A\in B\left( X\right) $ and $\Delta _{A}$ be the inner derivation on $%
B\left( X\right) ;$%
\begin{equation*}
\Delta _{A}:T\mapsto AT-TA,\text{ \ }T\in B\left( X\right) .
\end{equation*}%
We have%
\begin{equation*}
\Delta _{A}^{n}\left( T\right) =\sum\limits_{k=0}^{n}\left( -1\right) ^{k}%
\binom{n}{k}A^{n-k}TA^{k}\text{ \ }\left( n\in 
\mathbb{N}
\right) .
\end{equation*}

\begin{theorem}
For every operator $A\in B\left( X\right) $ with real spectrum we have 
\begin{equation*}
\mathcal{D}_{A}^{\alpha }\left( 
\mathbb{R}
\right) =\ker \Delta _{A}^{\left[ \alpha \right] +1}.
\end{equation*}%
In particular if $0\leq \alpha <1,$ then $\mathcal{D}_{A}^{\alpha }\left( 
\mathbb{R}
\right) =\left\{ A\right\} ^{\prime }.$
\end{theorem}

\begin{proof}
If $T\in \mathcal{D}_{A}^{\alpha }\left( 
\mathbb{R}
\right) ,$ then as 
\begin{equation*}
e^{tA}Te^{-tA}=e^{t\Delta _{A}}\left( T\right) ,
\end{equation*}%
we have%
\begin{equation*}
\left\Vert e^{t\Delta _{A}}\left( T\right) \right\Vert \leq C_{T}\left(
1+\left\vert t\right\vert \right) ^{\alpha }\text{, \ }\forall t\in 
\mathbb{R}
.
\end{equation*}%
By Proposition 2.2, $\sigma _{\Delta _{A}}\left( T\right) \subset i%
\mathbb{R}
.$ Further, since $\sigma \left( A\right) \subset 
\mathbb{R}
$ and 
\begin{equation*}
\sigma \left( \Delta _{A}\right) =\left\{ \lambda -\mu :\lambda ,\mu \in
\sigma \left( A\right) \right\}
\end{equation*}%
\cite[Theorem 3.5.1]{12}$,$ we have $\sigma \left( \Delta _{A}\right)
\subset 
\mathbb{R}
$. Therefore $\Delta _{A}$ has SVEP. On the other hand, as 
\begin{equation*}
\sigma _{\Delta _{A}}\left( T\right) \subseteq \sigma \left( \Delta
_{A}\right) \subset 
\mathbb{R}
,
\end{equation*}%
we get 
\begin{equation*}
\sigma _{\Delta _{A}}\left( T\right) \subseteq 
\mathbb{R}
\cap i%
\mathbb{R}
=\left\{ 0\right\} .
\end{equation*}%
Since $\Delta _{A}$ has SVEP, $\sigma _{\Delta _{A}}\left( T\right) \neq
\emptyset $ and so $\sigma _{\Delta _{A}}\left( T\right) =\left\{ 0\right\}
. $ Applying Corollary 2.8 to the operator $\Delta _{A}$ on the space $%
B\left( X\right) ,$ we obtain that%
\begin{equation*}
\Delta _{A}^{[\alpha ]+1}\left( T\right) =0.
\end{equation*}%
For the reverse inclusion, let $T\in B\left( X\right) $ and assume that
there exists $n\in 
\mathbb{N}
$ such that $\Delta _{A}^{n+1}\left( T\right) =0$ and $\Delta _{A}^{n}\left(
T\right) \neq 0.$ Then 
\begin{eqnarray*}
\left\Vert e^{tA}Te^{-tA}\right\Vert &=&\left\Vert e^{t\Delta _{A}}\left(
T\right) \right\Vert \\
&=&\left\Vert I+\frac{t\Delta _{A}\left( T\right) }{1!}+...+\frac{%
t^{n}\Delta _{A}^{n}\left( T\right) }{n!}\right\Vert \\
&=&O\left( 1+\left\vert t\right\vert \right) ^{n}\text{ \ }\left( \left\vert
t\right\vert \rightarrow \infty \right) .
\end{eqnarray*}%
This shows that $T\in \mathcal{D}_{A}^{n}\left( 
\mathbb{R}
\right) .$
\end{proof}

The following corollary generalizes the Williams result mentioned above.

\begin{corollary}
For every operator $A\in B\left( X\right) $ with one-point spectrum, we have 
\begin{equation*}
\mathcal{D}_{A}^{\alpha }\left( 
\mathbb{R}
\right) =\ker \Delta _{A}^{\left[ \alpha \right] +1}.
\end{equation*}%
In particular if $0\leq \alpha <1,$ then $\mathcal{D}_{A}^{\alpha }\left( 
\mathbb{R}
\right) =\left\{ A\right\} ^{\prime }.$
\end{corollary}

\begin{proof}
Assume that $\sigma \left( A\right) =\left\{ \lambda \right\} .$ If $T\in 
\mathcal{D}_{A}^{\alpha }\left( 
\mathbb{R}
\right) $ then $T\in \mathcal{D}_{B}^{\alpha }\left( 
\mathbb{R}
\right) ,$ where $B=A-\lambda I.$ Since $\sigma \left( B\right) =\left\{
0\right\} $ and $\Delta _{B}^{n}\left( T\right) =\Delta _{A}^{n}\left(
T\right) $ for all $n\in 
\mathbb{N}
$, by Theorem 3.1 we obtain our result.
\end{proof}

\begin{example}
If $\alpha \geq 1,$ then $\mathcal{D}_{A}^{\alpha }\left( 
\mathbb{R}
\right) \neq \left\{ A\right\} ^{\prime },$ in general. To see this, let $%
A=\left( 
\begin{array}{cc}
0 & 0 \\ 
1 & 0%
\end{array}%
\right) $ and $T=\left( 
\begin{array}{cc}
1 & 0 \\ 
0 & 0%
\end{array}%
\right) $ be two $2\times 2$ matrices on $2-$dimensional Hilbert space. As $%
A^{2}=0,$ we have $\sigma \left( A\right) =\left\{ 0\right\} $ and 
\begin{equation*}
e^{tA}Te^{-tA}=\left( I+tA\right) T\left( I-tA\right) =\left( 
\begin{array}{cc}
1 & 0 \\ 
t & 0%
\end{array}%
\right) \text{, \ }\forall t\in 
\mathbb{R}
.
\end{equation*}%
Since%
\begin{equation*}
\left\Vert e^{tA}Te^{-tA}\right\Vert =\left( 1+\left\vert t\right\vert
^{2}\right) ^{\frac{1}{2}},
\end{equation*}%
we have $T\in \mathcal{D}_{A}^{1}\left( 
\mathbb{R}
\right) ,$ but $AT\neq TA.$
\end{example}

\begin{example}
The condition $\left\Vert e^{tA}Te^{-tA}\right\Vert \leq C_{T}\left(
1+t\right) ^{\alpha }$ for all $t\geq 0$ is not sufficient in Corollary 3.2.
To see this, let $S$ be a contraction on a Hilbert space $H$ such that $%
\sigma \left( S\right) =\left\{ 1\right\} $ and $S\neq I$ $($see, Example
2.10$)$. If $Q:=S-I,$ then $\sigma \left( Q\right) =\left\{ 0\right\} $ and $%
\left\Vert e^{tQ}\right\Vert \leq 1$ for all $t\geq 0.$ Define the operator $%
A$ on $H\oplus H$ by $A=\left( 
\begin{array}{cc}
-Q & 0 \\ 
0 & 0%
\end{array}%
\right) $. Since $e^{tA}=\left( 
\begin{array}{cc}
e^{-tQ} & 0 \\ 
0 & I%
\end{array}%
\right) ,$ for $T=\left( 
\begin{array}{cc}
0 & 0 \\ 
I & 0%
\end{array}%
\right) $, we have 
\begin{equation*}
e^{tA}Te^{-tA}=\left( 
\begin{array}{cc}
0 & 0 \\ 
e^{tQ} & 0%
\end{array}%
\right) .
\end{equation*}%
It follows that 
\begin{equation*}
\left\Vert e^{tA}Te^{-tA}\right\Vert =\left\Vert e^{tQ}\right\Vert \leq 1%
\text{ }\left( \forall t\geq 0\right) ,
\end{equation*}%
but $AT\neq TA.$
\end{example}

Next, we give some results related to the decomposability.

\begin{proposition}
Assume that $A\in B\left( X\right) $ is decomposable and $T\in B\left(
X\right) $ satisfies the condition 
\begin{equation*}
\left\Vert e^{tA}Te^{-tA}\right\Vert \leq C\left( 1+\left\vert t\right\vert
\right) ^{\alpha }\text{,}
\end{equation*}
for all $t\in 
\mathbb{R}
$ and for some constant $C>0.$ Then the following conditions are equivalent:

$\left( a\right) $ $TX_{A}\left( F\right) \subseteq X_{A}\left( F\right) $
for every closed set $F\subset 
\mathbb{C}
.$

$\left( b\right) \mathcal{\ }\Delta _{A}^{\left[ \alpha \right] +1}\left(
T\right) =0.$

In particular, if $0\leq \alpha <1,$ then $AT=TA$ if and only if $%
TX_{A}\left( F\right) \subseteq X_{A}\left( F\right) $ for every closed set $%
F\subset 
\mathbb{C}
.$
\end{proposition}

\begin{proof}
(a)$\Rightarrow $(b) We have%
\begin{equation*}
\left\Vert e^{t\Delta _{A}}\left( T\right) \right\Vert =\left\Vert
e^{tA}Te^{-tA}\right\Vert \leq C\left( 1+\left\vert t\right\vert \right)
^{\alpha }\text{, \ }\forall t\in 
\mathbb{R}
.
\end{equation*}%
Since $A$ is decomposable, $\Delta _{A}$ has SVEP \cite[Proposition 3.4.6]%
{12} and therefore,%
\begin{equation*}
r_{\Delta _{A}}\left( T\right) =\underset{n\rightarrow \infty }{\overline{%
\lim }}\left\Vert \Delta _{A}^{n}\left( T\right) \right\Vert ^{\frac{1}{n}}.
\end{equation*}%
On the other hand, $TX_{A}\left( F\right) \subseteq X_{A}\left( F\right) $
for every closed set $F\subset 
\mathbb{C}
$ if and only if 
\begin{equation*}
\lim_{n\rightarrow \infty }\left\Vert \Delta _{A}^{n}\left( T\right)
\right\Vert ^{\frac{1}{n}}=0
\end{equation*}%
\cite[Corollary 3.4.5]{12}. Consequently, $r_{\Delta _{A}}\left( T\right) =0$
and so $\sigma _{\Delta _{A}}\left( T\right) =\left\{ 0\right\} .$ By
Corollary 2.8, $\Delta _{A}^{\left[ \alpha \right] +1}\left( T\right) =0.$

(b)$\Rightarrow $(a) If $\Delta _{A}^{\left[ \alpha \right] +1}\left(
T\right) =0,$ then $\lim_{n\rightarrow \infty }\left\Vert \Delta
_{A}^{n}\left( T\right) \right\Vert ^{\frac{1}{n}}=0.$ By \cite[Corollary
3.4.5]{12}, $TX_{A}\left( F\right) \subseteq X_{A}\left( F\right) $ for
every closed set $F\subset 
\mathbb{C}
.$
\end{proof}

\begin{remark}
The proof of $($b$)\Rightarrow \left( \text{a}\right) $ in Proposition 3.5,
can be simplified as follows: It suffices to show that $\sigma _{A}\left(
Tx\right) \subseteq \sigma _{A}\left( x\right) $ for every $x\in X.$ If $%
x\in X$ and $\lambda \in \rho _{A}\left( x\right) ,$ then there is a
neighborhood $U_{\lambda }$ of $\lambda $ with $u\left( z\right) $ analytic
on $U_{\lambda }$ having values in $X$ such that $\left( zI-A\right) u\left(
z\right) =x$ for all $z\in U_{\lambda }.$ Using this, it is easy to check
that the function 
\begin{equation*}
v\left( z\right) :=Tu\left( z\right) -\Delta _{A}\left( T\right) u^{\prime
}\left( z\right) +...+\left( -1\right) ^{k}\Delta ^{k}\left( T\right)
u^{\left( k\right) }\left( z\right) \text{ \ }\left( k=\left[ \alpha \right]
\right)
\end{equation*}%
satisfies the equation $\left( zI-A\right) v\left( z\right) =Tx$ for all $%
z\in U_{\lambda }.$ This shows that $\lambda \in \rho _{A}\left( Tx\right) .$
\end{remark}

\begin{proposition}
Assume that the operators $A,T$ in $B\left( X\right) $ satisfy the following
conditions:

$\left( i\right) $ $A$ is decomposable and $\sigma \left( A\right) \subset 
\mathbb{C}
\diagdown 
\mathbb{R}
_{-},$ where $%
\mathbb{R}
_{-}=\left\{ t\in 
\mathbb{R}
:t\leq 0\right\} .$

$\left( ii\right) $ $\left\Vert A^{n}TA^{-n}\right\Vert \leq C\left(
1+\left\vert n\right\vert \right) ^{\alpha }$ $\left( 0\leq \alpha <1\right) 
$ for all $n\in 
\mathbb{Z}
$ and for some constant $C>0$.

Then $AT=TA$ if and only if $TX_{A}\left( F\right) \subseteq X_{A}\left(
F\right) $ for every closed set $F\subset 
\mathbb{C}
$.
\end{proposition}

\begin{proof}
It is trivial to show that if $AT=TA,$ then $\sigma _{A}\left( Tx\right)
\subseteq \sigma _{A}\left( x\right) $ for every $x\in X.$ It follows that $%
TX_{A}\left( F\right) \subseteq X_{A}\left( F\right) $ for every closed set $%
F\subset 
\mathbb{C}
$. Now, assume that $TX_{A}\left( F\right) \subseteq X_{A}\left( F\right) $
for every closed set $F\subset 
\mathbb{C}
$. For $z\in 
\mathbb{C}
\diagdown 
\mathbb{R}
_{-},$ let $f\left( z\right) :=\log z$ is the principal branch of the
logarithm. Then $A=e^{B},$ where $B=\log A:=f\left( A\right) .$ Notice that $%
B$ is decomposable \cite[Theorem 3.3.6]{12}. Moreover, if $t\in 
\mathbb{R}
,$ then as $t=n+r,$ where $n\in 
\mathbb{Z}
$, $\left\vert r\right\vert <1,$ and $\left\vert n\right\vert \leq
\left\vert t\right\vert ,$ we can write%
\begin{eqnarray*}
\left\Vert e^{tB}Te^{-tB}\right\Vert  &=&\left\Vert
e^{rB}e^{nB}Te^{-nB}e^{-rB}\right\Vert  \\
&\leq &C\left\Vert e^{B}\right\Vert \left\Vert e^{-B}\right\Vert \left(
1+\left\vert n\right\vert \right) ^{\alpha } \\
&\leq &C\left\Vert e^{B}\right\Vert \left\Vert e^{-B}\right\Vert \left(
1+\left\vert t\right\vert \right) ^{\alpha }.
\end{eqnarray*}%
Since 
\begin{equation*}
X_{B}\left( F\right) =X_{A}\left( f^{-1}\left( F\right) \right) ,
\end{equation*}%
for every closed set $F\subset 
\mathbb{C}
$ \cite[Theorem 3.3.6]{12}, we get%
\begin{equation*}
TX_{B}\left( F\right) =TX_{A}(f^{-1}\left( F\right) \subseteq
X_{A}(f^{-1}\left( F\right) =X_{B}\left( F\right) .
\end{equation*}%
By Proposition 3.6, $BT=TB$ which implies $e^{B}T=Te^{B}.$ Hence $AT=TA.$
\end{proof}

\end{document}